\documentclass[a4paper,12pt]{article}
\usepackage{hyperref}
\hypersetup{
    colorlinks=true,
    linktoc=all,
    linkcolor=red,     
    citecolor=blue,     
}
\usepackage{graphicx}
\usepackage{amsmath,amssymb,amsfonts,amsthm,graphics,enumitem}
\usepackage{latexsym, bm}
\usepackage{multicol}
\usepackage{indentfirst}

\usepackage{setspace}
\newtheorem{theorem}{Theorem}[section]
\newtheorem{lemma}{Lemma}[section]

\newtheorem{corollary}{Corollary}[section]

\newtheorem{definition}{Definition}[section]
\newtheorem{problem}{Problem}[section]

\newtheorem{remark}{Remark}[section]

\newcommand{\ignore}[1]{}
\newtheorem{claimnum}{Claim}
\newtheorem{claimrom}{Claim}

\begin{document}
\noindent
\begin{spacing}{1}

\title {Rainbow panconnectivity in a graph collection}
\date{}

\author{
Menghan Ma\footnote{Center for Discrete Mathematics, Fuzhou University,
Fuzhou, 350108, P.~R.~China. Email: {\tt menghanma664@163.com}. } \;\;   \;\; Lihua You\footnote{School of Mathematical Sciences, South China Normal University, Guangzhou, 510631, P. R. China.
E-mail: {\tt ylhua@scnu.edu.cn}.} \;\;   \;\; Xiaoxue Zhang\footnote{Corresponding author. School of Mathematical Sciences, South China Normal University, Guangzhou, 510631, P. R. China.
E-mail: {\tt zhang\_xx1209@163.com}.}
}
\maketitle
\begin{abstract}
Let $\mathbf{G}=\{G_1,\dots,G_{n-1}\}$ be a collection of not necessarily distinct $n$-vertex graphs with the same vertex set $V$. A path $P$ with $V(P)\subseteq V$ and $|E(P)|\leq n-1$ is called \emph{rainbow} in $\mathbf{G}$, if there exists an injection $\phi\colon E(P)\to [n-1]$ such that $e\in E(G_{\phi(e)})$ for each $e\in E(P)$.
The graph collection $\mathbf{G}$ is said to be \emph{rainbow panconnected} if for every pair of vertices $x,y\in V$, there exists a rainbow path of $k$ vertices joining $x$ and $y$ in $\mathbf{G}$ for every integer $k\in \left[d_{\mathbf{G}}(x,y)+1, n\right]$, where $d_{\mathbf{G}}(x,y)$ is the length of a shortest rainbow path between $x$ and $y$ in $\mathbf{G}$.
In this paper, we study the rainbow panconnectivity of $\mathbf{G}$ under the minimum degree condition. Our result improves upon the corresponding results of [J. Graph Theory, \textbf{104}(2)(2023), 341--359] and [Electron. J. Combin., \textbf{32}(4)(2025), \#P4.17].


\end{abstract}
\section{Introduction}
In recent years, transversals in graph collections have become a popular research topic. This concept was first formally introduced by Joos and Kim \cite{Joos} in $2020$. Let $\mathbf{G}=\{G_1,\dots,G_s\}$ be a collection of not necessarily distinct graphs with the same vertex set $V$.
For a graph $H$ with $V(H)\subseteq V$ and $|E(H)|\leq s$, if there exists an injection $\phi\colon E(H)\to [s]$ such that $e\in E(G_{\phi(e)})$ for each $e\in E(H)$, then we say that $H$ is a \emph{partial $\mathbf{G}$-transversal} if $|E(H)|<s$, $H$ is a \emph{$\mathbf{G}$-transversal} if $|E(H)|=s$. That is to say, each edge of $H$ comes from a distinct graph $G_i$ with $i\in[s]$, we further say that $H$ is \emph{rainbow} in $\mathbf{G}$.
The following question was proposed by Joos and Kim in \cite{Joos}.

\begin{problem}[\upshape\cite{Joos}]\label{prob1-1}
Let $H$ be a graph with $m$ edges and $\mathbf{G}=\{G_1,\dots,G_s\}$ be a collection of not necessarily distinct graphs with the same vertex set $V$, where $m\leq s$. Which properties imposed on $\mathbf{G}$ yield a rainbow copy of $H$?
\end{problem}

Focusing on Problem \ref{prob1-1}, there has been a great deal of work done on it. Scholars have imposed a variety of constraints on graph collections, including Dirac-type conditions \cite{Bradshaw,Cheng,Hu,Joos,Li2023,Li2024,Ma,Sun}, Ore-type conditions \cite{Li2023,LiandWang,LiandLuo,LiuandChen}, edge number conditions \cite{Aharoni,Babinski,Liu,Zhang} and so on. For more results, we refer the readers to a survey \cite{SunandWang}.

A path (cycle) on $l$ vertices is referred to as an $l$-path (a $l$-cycle), respectively. A path (cycle) which contains all vertices of a graph is called a \emph{Hamiltonian path (cycle)}.
A graph $G$ is said to be \emph{panconnected}, if for every pair of vertices $x,y\in V(G)$, there exists a $k$-path joining $x$ and $y$ in $G$ for every integer $k\in \left[d_{G}(x,y)+1, n\right]$. The minimum degree condition that guarantees a graph is panconnected was given by Williamson \cite{Williamson}.

\begin{theorem}[\upshape\cite{Williamson}]\label{thm1-1}
If $G$ is a graph of order $n\geq4$ such that $\delta(G)\geq \frac{n+2}{2}$, then $G$ is panconnected.
\end{theorem}

In 2023, Li et al. \cite{Li2023} generalized Theorem \ref{thm1-1} to the setting of graph transversals, i.e. the rainbow version of Theorem \ref{thm1-1}.
A collection $\mathbf{G}$ of $n$-vertex graphs with the same vertex set $V$ is said to be \emph{rainbow panconnected} if for every pair of vertices $x,y\in V$, there exists a rainbow $k$-path joining $x$ and $y$ in $\mathbf{G}$ for every integer $k\in \left[d_{\mathbf{G}}(x,y)+1, n\right]$, where $d_{\mathbf{G}}(x,y)$ is the length of a shortest rainbow path between $x$ and $y$ in $\mathbf{G}$.
In \cite{Li2023}, they also defined a special graph collection as follows, and obtained Theorem \ref{thm1-3}.

\begin{definition}[\upshape\cite{Li2023}]\label{def1-2}
Let $n$ be odd, $Q_1$ be an empty graph with $\frac{n-1}{2}$ vertices and $Q_2$ be a graph on $\frac{n+1}{2}$ vertices with $\delta(Q_2)\geq 1$ such that one component of $Q_2$ is a single edge. Define $\mathbf{F_m}=\{F_1,\dots,F_{m}\}$ as a collection of $n$-vertex graphs with the same vertex set such that $F_i=Q_1 \vee Q_2$ for each $i\in [m]$.
\end{definition}

\begin{theorem}[\upshape\cite{Li2023}]\label{thm1-3}
Let $\mathbf{G}=\{G_1, \dots, G_{n}\}$ be a collection of not necessarily distinct $n$-vertex graphs with the same vertex set $V$. If $\delta(G_i)\geq\frac{n+1}{2}$ for each $i\in [n]$, then either $\mathbf{G}$ is rainbow panconnected or $\mathbf{G}=\mathbf{F_n}$.
\end{theorem}

Clearly, the edges of a rainbow path in Theorem \ref{thm1-3} only come from at most $n-1$ graphs in the collection $\mathbf{G}=\{G_1, \dots, G_{n}\}$. In $2025$, Sun et al. \cite{Sun} established the same minimum degree condition for a collection of $n-1$ graphs to be rainbow panconnected.

\begin{theorem}[\upshape\cite{Sun}]\label{thm1-4}
Let $n\geq 4$ be an even integer or a sufficiently large odd integer, $\mathbf{G}=\{G_1, \dots, G_{n-1}\}$ be a collection of not necessarily distinct $n$-vertex graphs with the same vertex set $V$. If $\delta(G_i)\geq\frac{n+1}{2}$ for each $i\in [n-1]$, then either $\mathbf{G}$ is rainbow panconnected or $\mathbf{G}=\mathbf{F_{n-1}}$.
\end{theorem}

However, Theorem \ref{thm1-4} requires $n$ to be sufficiently large when $n$ is an odd integer. The authors in \cite{Sun} conjectured that Theorem \ref{thm1-4} holds for all $n$. Motivated by this, we confirm the conjecture and obtain the following result.

\begin{theorem}\label{thm1-5}
Let $n\geq3$ and $\mathbf{G}=\{G_1, \dots, G_{n-1}\}$ be a collection of not necessarily distinct $n$-vertex graphs with the same vertex set $V$. If $\delta(G_i)\geq\frac{n+1}{2}$ for each $i\in [n-1]$, then either $\mathbf{G}$ is rainbow panconnected or $\mathbf{G}=\mathbf{F_{n-1}}$.
\end{theorem}

Here, we introduce some basic notation and concepts in graph theory. In this paper, we only consider finite and simple graphs.
Let $G=(V(G),E(G))$ be a simple graph with vertex set $V(G)$ and edge set $E(G)$.
For a vertex $u\in V(G)$, the set of neighbours of a vertex $u$ in $G$ is denoted by $N_{G}(u)$, and $d_{G}(u)=|N_{G}(u)|$ is the \emph{degree} of $u$ in $G$.
Let $\delta(G)=\min\{d_{G}(u): u\in V(G)\}$ denote the \emph{minimum degree} of $G$ and $\sigma_2(G)=\min\{d_{G}(u)+d_{G}(v): u,v\in V(G),\, uv\notin E(G)\}$ denote the \emph{minimum degree sum} of any pair of nonadjacent vertices in $G$.
For any two vertices $u,v\in V(G)$, $d_{G}(u,v)$ is defined as the length of a shortest path between $u$ and $v$. For two disjoint graphs $G_1$ and $G_2$, the \emph{join} of $G_1$ and $G_2$, denoted by $G_1\vee G_2$, is the graph obtained from $G_1$ and $G_2$ by adding all edges between $V(G_1)$ and $V(G_2)$.
For $U\subseteq V(G)$, the \emph{induced subgraph} $G[U]$ satisfies $V(G[U])=U$ and $E(G[U])=\{uv\in E(G): u,v\in U\}$.
For two disjoint subsets $U,W\subset V(G)$, $G[U,W]$ denotes the bipartite subgraph of $G$ with vertex bipartition $(U,W)$ and edge set $E(G[U,W])=\{uw\in E(G): u\in U,w\in W$\}. We use $E_G[U,W]$ for $E(G[U,W])$ for short.
For positive integers $a<b$, let $[a]=\{1,2,\dots,a\}$ and $[a,b]=\{a,a+1,\dots,b\}$. For any terminology and notation that are not explicitly defined, we refer the reader to \cite{Bondy}.

\section{Preliminaries}
In this section, we present some necessary definitions and known results that will be used later.

For a path $P=u_1u_2\cdots u_s$, we denote the subpath of $P$ between $u_i$ and $u_j$ by $u_iPu_j$, where $i,j\in[s]$. We use $|P|$ to represent the number of vertices of path $P$. For two distinct vertices $v_i,v_j$ on a cycle $C=v_1v_2\cdots v_sv_1$ with $i<j$, we denote the two paths between them as the forward segment $v_iCv_j=v_iv_{i+1}\cdots v_j$ and the backward segment $v_iC^{-}v_j=v_iv_{i-1}\cdots v_j$, where subscripts are taken modulo $s$.
For a subgraph $H$ of $G$ and a vertex $u\in V(G)\setminus V(H)$, let $N_{G}(u,H)=N_{G}(u)\cap V(H)$ and $d_{G}(u,H)=|N_{G}(u,H)|$.

For a collection $\mathbf{G}$ of $n$-vertex graphs with the same vertex set $V$, we say $\mathbf{G}$ is \emph{rainbow Hamiltonian connected} if for every pair of vertices $x,y\in V$, there exists a rainbow Hamiltonian path joining $x$ and $y$ in $\mathbf{G}$. Sun et al. \cite{Sun} gave a minimum degree condition to guarantee that a graph collection is rainbow Hamiltonian connected.

\begin{theorem}[\upshape\cite{Sun}]\label{thm2-1}
Let $n\geq 3$ and $\mathbf{G}=\{G_1, \dots, G_{n-1}\}$ be a collection of not necessarily distinct $n$-vertex graphs with the same vertex set $V$. If $\delta(G_i)\geq\frac{n+1}{2}$ for each $i\in [n-1]$, then $\mathbf{G}$ is rainbow Hamiltonian connected.
\end{theorem}

Li et al. \cite{Li2023} studied the existence of rainbow Hamiltonian path.

\begin{theorem}[\upshape\cite{Li2023}]\label{thm2-2}
Let $\mathbf{G}=\{G_1, \dots, G_{n}\}$ be a collection of not necessarily distinct $n$-vertex graphs with the same vertex set $V$. If $\sigma_2(G_i)\geq n-2$ for each $i\in [n]$, then one of the following statements holds:
\begin{enumerate}[label=(\roman*), font=\upshape, itemsep=2pt, align=left, leftmargin=1em]
 \item $\mathbf{G}$ has a rainbow Hamiltonian path;
 \item $G_1=\dots=G_{n}=K_\ell \cup K_{n-\ell}$, where $\ell \in [n-1]$;
 \item $n$ is even and there is a partition $(H, I)$ of $V$ with $|H|=\frac{n-2}{2}$ and $|I|=\frac{n+2}{2}$. For each $i\in [n]$, $G_i=G_i[H]\vee G_i[I]$, where $G_i[I]$ is an independent set and $G_i[H]$ is an arbitrary graph.
\end{enumerate}
\end{theorem}

Note that $\delta(G_i)\geq\frac{n}{2}-1$ implies $\sigma_2(G_i)\geq n-2$ for each $i\in[n]$. Therefore, Theorem \ref{thm2-2} still holds when we replace its condition with $\delta(G_i)\geq\frac{n}{2}-1$.
In particular, condition $\delta(G_i)\geq\frac{n}{2}-1$ implies that for (ii) of Theorem \ref{thm2-2}, we have $\ell-1\geq\frac{n}{2}-1$ and $n-\ell-1\geq \frac{n}{2}-1$, which yields $\ell=n-\ell=\frac{n}{2}$, i.e., $n$ is even.
Therefore, the following corollary holds.

\begin{corollary}\label{cor2-3}
Let $\mathbf{G}=\{G_1, \dots, G_{n}\}$ be a collection of not necessarily distinct $n$-vertex graphs with the same vertex set $V$. If $\delta(G_i)\geq\frac{n}{2}-1$ for each $i\in [n]$, then one of the following statements holds:
\begin{enumerate}[label=(\roman*), font=\upshape, itemsep=2pt, align=left, leftmargin=1em]
\item $\mathbf{G}$ has a rainbow Hamiltonian path;
\item $n$ is even and $G_1=\dots=G_{n}=K_{\frac{n}{2}}\cup K_{\frac{n}{2}}$;
\item $n$ is even and there is a partition $(H, I)$ of $V$ with $|H|=\frac{n-2}{2}$ and $|I|=\frac{n+2}{2}$. For each $i\in [n]$, $G_i=G_i[H]\vee G_i[I]$, where $G_i[I]$ is an independent set and $G_i[H]$ is an arbitrary graph.
\end{enumerate}
\end{corollary}

\section{Proof of Theorem \ref{thm1-5}}
This section will be dedicated to proving Theorem \ref{thm1-5}.

Since Theorem \ref{thm1-4} implies that Theorem \ref{thm1-5} holds when $n$ is even, and the conclusion is trivial for $n=3$, we may assume that $n(\geq5)$ is an odd integer in the following proof.

\begin{lemma}\label{l1}
For any two vertices $x,y\in V$, there exists a rainbow $k$-path joining $x$ and $y$ in $\mathbf{G}$ for $k\in\{3,n\}$, and $k=2$ if $d_{\mathbf{G}}(x,y)=1$.
\end{lemma}

\begin{proof}
For any two vertices $x,y\in V$, there exists a rainbow $2$-path joining $x$ and $y$ if $d_{\mathbf{G}}(x,y)=1$, and Theorem \ref{thm2-1} implies the existence of a rainbow $n$-path between $x$ and $y$ since $\delta(G_i)\geq\frac{n+1}{2}$ for each $i\in [n-1]$.

Since $\delta(G_i)\geq\frac{n+1}{2}$ for each $i\in [n-1]$, we have $N_{G_1}(x)\cap N_{G_2}(y)\neq \emptyset$, which implies that there exists a rainbow $3$-path joining $x$ and $y$.
\end{proof}

By Lemma \ref{l1}, to prove Theorem \ref{thm1-5}, it suffices to show that one of the following two statements holds:\\
$(\mathrm{R_1})$ For any $x$ and $y$, there exists a rainbow $k$-path joining $x$ and $y$ in $\mathbf{G}$ for every integer $k\in[4,n-1]$;\\
$(\mathrm{R_2})$ $\mathbf{G}=\mathbf{F_{n-1}}$.

\begin{lemma}\label{lem1}
If $n=5$, then $\mathbf{G}$ is rainbow panconnected.
\end{lemma}

\begin{proof}
By Lemma \ref{l1} and the definition of rainbow panconnected, we only need to show $(\mathrm{R_1})$ holds.

For every $i\in [4]$, we obtain $\delta(G_i)\geq\frac{n+1}{2}\geq 3$ from the assumption $n=5$. Then $\delta(G_i-\{x,y\})\geq 1$. Let $V\setminus\{x,y\}=\{u_1,u_2,u_3\}$. Without loss of generality, let $u_1u_2\in E(G_4)$. Since $d_{G_1}(x)\geq3$, we have $N_{G_1}(x)\cap \{u_1,u_2\}\neq \emptyset$, and we may assume $xu_1\in E(G_1)$.

If $u_2y\in E(G_2)\cup E(G_3)$, then $xu_1u_2y$ is a rainbow $4$-path, $(\mathrm{R_1})$ holds.

If $u_2y\notin E(G_2)\cup E(G_3)$, then $N_{G_2}(u_2)=N_{G_3}(u_2)=\{x,u_1,u_3\}$. Since $d_{G_4}(y)\geq3$, we have $N_{G_4}(y)\cap \{u_2,u_3\}\neq \emptyset$, which implies $xu_1u_2y$ or $xu_2u_3y$ is a rainbow $4$-path, $(\mathrm{R_1})$ holds.
\end{proof}

In the following, we consider $n\geq7$. Let $U=V\setminus\{x,y\}$.

If $ux\in E(G_i)$ for any $u\in U$ and any $i\in [n-1]$, then $\mathbf{G}$ has a rainbow $k$-path joining $x$ and $y$ for $k\in[4,n-1]$ since there exists a rainbow Hamiltonian path joining $x$ and $y$ by Theorem \ref{thm2-1}, and thus $(\mathrm{R_1})$ holds.

Otherwise, there exists $u\in U$ such that $ux\notin E(G_i)$ for some $i\in [n-1]$. Without loss of generality, we assume $zx\notin E(G_{n-1})$ with $z\in U$.
Let $H_i=G_i-\{x,y,z\}$ for each $i\in[n-2]$ and $\mathbf{H}=\{H_1,\dots,H_{n-2}\}$. Let $\mathbf{H_1},\dots,\mathbf{H_{n-2}}$ be $n-2$ graph collections with $\mathbf{H_j}=\mathbf{H}\setminus\{H_j\}$ for $j\in[n-2]$. Now we study $\mathbf{H_1},\dots,\mathbf{H_{n-2}}$ by Lemmas \ref{lem2}-\ref{lem8}.

\begin{lemma}\label{lem2}
If there exists some $j\in[n-2]$ such that $\mathbf{H_j}$ contains a rainbow Hamiltonian cycle, then $\mathbf{G}$ is rainbow panconnected.
\end{lemma}

\begin{proof}
Without loss of generality, let $j=n-2$ and $C=u_1u_2\cdots u_{n-3}u_1$ be a rainbow Hamiltonian cycle in $\mathbf{H_{n-2}}$ with $u_iu_{i+1}\in E(G_i)$ for each $i\in [n-3]$, where the subscripts are taken modulo $n-3$. Then we consider the following sets for any $k\in[4,n-1]$:
\[
I_k=\{i\in[n-3]:xu_{i+k-3}\in E(G_{n-1})\}, \quad I_0=\{i\in[n-3]:yu_i\in E(G_{n-2})\}.
\]
Since $xz\notin E(G_{n-1})$ and $\delta(G_i)\geq\frac{n+1}{2}$ for each $i\in [n-1]$, we have $|I_{k}|=|N_{G_{n-1}}(x)\cap V(C)|=d_{G_{n-1}}(x,C)\geq d_{G_{n-1}}(x)-1\geq\frac{n-1}{2}$ and $|I_{0}|=d_{G_{n-2}}(y,C)\geq d_{G_{n-2}}(y)-2\geq\frac{n-3}{2}$, and thus $|I_{k}|+|I_{0}|\geq n-2$. In addition, $|I_{k}\cup I_{0}|\leq n-3$, which implies $I_{k}\cap I_{0}\neq\emptyset$. Let $s\in I_{k}\cap I_{0}$. Then $xu_{s+k-3}C^{-}u_sy$ is the a rainbow $k$-path joining $x$ and $y$, and thus $(\mathrm{R_1})$ holds, which implies $\mathbf{G}$ is rainbow panconnected.
\end{proof}

Now, we assume that $\mathbf{H_j}$ contains no rainbow Hamiltonian cycles for any $j\in[n-2]$.

\begin{lemma}\label{lem3}
If $\mathbf{H_j}$ contains a rainbow $(n-4)$-cycle for some $j\in[n-2]$, then $\mathbf{G}$ is rainbow panconnected.
\end{lemma}

\begin{proof}
Without loss of generality, assume that $j=n-2$ and $C=u_1u_2\cdots u_{n-4}u_1$ is a rainbow $(n-4)$-cycle in $\mathbf{H_{n-2}}$ with $u_iu_{i+1}\in E(G_i)$ for each $i\in [n-4]$, where the subscripts are taken modulo $n-4$.

Let $\{w\}=V\setminus(V(C)\cup\{x,y,z\})$. We define the following two sets:
\[
A=\{s\in[n-4]:wu_{s+1}\in E(G_{n-3})\}, \quad B=\{s\in[n-4]:wu_s\in E(G_{n-2})\}.
\]
Then $|A|\geq\frac{n-5}{2}$ and $|B|\geq\frac{n-5}{2}$ since $d_{G_i}(w,C)\geq\frac{n+1}{2}-3=\frac{n-5}{2}$ for each $i\in\{n-3,n-2\}$.

If $A\cap B\neq \emptyset$, we assume $s\in A\cap B$, then $wu_sC^{-}u_{s+1}w$ is a rainbow Hamiltonian cycle in $\mathbf{H_s}$, a contradiction. So $A\cap B=\emptyset$, and
\begin{equation}\tag{$1$}\label{e1}
n-5=\frac{n-5}{2}+\frac{n-5}{2} \leq|A|+|B|=|A\cup B|\leq n-4=|V(C)|.
\end{equation}

\begin{claimrom}\label{claim1}
$|A|=|B|=\frac{n-5}{2}$.
\end{claimrom}

\begin{proof}
If not, one of $A$ and $B$ has size $\frac{n-3}{2}$ and the other has size $\frac{n-5}{2}$ by \eqref{e1}. We assume $|A|=\frac{n-3}{2}$ and $|B|=\frac{n-5}{2}$. The proof for the case $|A|=\frac{n-5}{2}$ and $|B|=\frac{n-3}{2}$ is analogous, so we omit it.

Since $|A|+|B|=|V(C)|$, $|A|=|B|+1$ and $A\cap B=\emptyset$, we conclude that $A$ must contain at least two consecutive integers (all integers mentioned herein are taken modulo $n-4$). Choose a maximal consecutive subset $\{s,s+1,\dots,s+t\}\subseteq A$ such that $s-1,s+t+1\notin A$, where $t\geq1$.
Then $s-1,s+t+1\in B$ since $|A\cup B|=|V(C)|$ and $A\cap B=\emptyset$, which implies that $wu_{s+t}\in E(G_{n-3})$ and $wu_{s+t+1}\in E(G_{n-2})$. Thus $wu_{s+t}C^{-}u_{s+t+1}w$ is a rainbow Hamiltonian cycle in $\mathbf{H_{s+t}}$, a contradiction.

This completes the proof of Claim \ref{claim1}.
\end{proof}

By Claim \ref{claim1}, we have $|A\cup B|=n-5$ and $\{x,y,z\}\subseteq N_{G_i}(w)$ for each $i\in\{n-3,n-2\}$ since $\delta(G_i)\geq\frac{n+1}{2}$.

\begin{claimrom}\label{claim2}
$A,B$ contain no consecutive integers modulo $n-4$.
\end{claimrom}

\begin{proof}
We prove that $A$ contains no consecutive integers modulo $n-4$, and the case for $B$ is analogous, which we omit. All integers involved in this proof are considered modulo $n-4$.

To the contrary, $A$ contains consecutive integers, and assume that $\{a,a+1,\dots,a+m\}\subseteq A$ is a maximal consecutive subset of $A$ such that $a-1,a+m+1\notin A$, where $m\geq1$.

Firstly, we show $A\cup B=[n-4]\setminus\{a+m+1\}$. If $a+m+1\in B$, then $wu_{a+m+1}\in E(G_{n-2})$, and $wu_{a+m}C^{-}u_{a+m+1}w$ is a rainbow Hamiltonian cycle in $\mathbf{H_{a+m}}$, a contradiction. Therefore, $a+m+1\notin A\cup B$, and $A\cup B=[n-4]\setminus\{a+m+1\}$ since $|A\cup B|=n-5$ and $A\cap B=\emptyset$. This implies that $a-1\in B$ since $a-1\notin A$.

Similar to the above proof, $A$ contains no other consecutive subset besides $\{a,a+1,\dots,a+m\}$. Otherwise, there exists $t\in[n-4]$ with $t\neq a+m+1$ such that $t\notin A\cup B$, meaning that $|A|+|B|\le n-6$, a contradiction to \eqref{e1}.

Secondly, we show $a+m+2\in A$. If $a+m+2\in B$, by similar analysis and $wu_{a+m+1}\in E(G_{n-3})$, we have $wu_{a+m+1}C^{-}u_{a+m+2}w$ is a rainbow Hamiltonian cycle in $\mathbf{H_{a+m+1}}$, a contradiction. Thus $a+m+2\in A$ since $A\cup B=[n-4]\setminus\{a+m+1\}$.

Thirdly, $B$ contains no consecutive integers. Otherwise, let $\{b,b+1,\dots,b+m'\}\subseteq B$ be a maximal consecutive subset of $B$, where $m'\geq1$. Then $b-1,b+m'+1\notin B$. Clearly, $b-1\neq a+m+1$ and $b+m'+1\neq a+m+1$ since $\{a+m,a+m+2\}\subseteq A$ while $\{b,b+m'\}\subseteq B$. Then $b-1,b+m'+1\in A$ since $A\cup B=[n-4]\setminus\{a+m+1\}$, which implies $wu_{b}C^{-}u_{b+1}w$ is a rainbow Hamiltonian cycle in $\mathbf{H_b}$, a contradiction.

Finally, combining the above arguments, $|A\setminus\{a,a+1,\dots,a+m\}|=|B|$, and thus $|A|-|B|=m+1\geq2$, contradicting $|A|=|B|$.
\end{proof}

\begin{claimrom}\label{claim3}
If $[n-4]\setminus(A\cup B)=\{s\}$, then $s-1\in B$ and $s+1\in A$.
\end{claimrom}

\begin{proof}
If $s+1\in B$, then $s-1\in A$ by Claim \ref{claim2}. Thus $wu_{s}\in E(G_{n-3})$ and $wu_{s+1}\in E(G_{n-2})$, which implies $wu_{s}C^{-}u_{s+1}w$ is a rainbow Hamiltonian cycle in $\mathbf{H_{s}}$, a contradiction.
Therefore, $s+1\in A$, and thus $s-1\in B$ by Claim \ref{claim2}.
\end{proof}

From the definitions of $A$, $B$, Claim \ref{claim3}, and the fact that $\{x,y,z\}\subseteq N_{G_i}(w)$ for each $i\in\{n-3,n-2\}$, we have $N_{G_{n-3}}(w)=N_{G_{n-2}}(w)$.
By Claim \ref{claim2}, without loss of generality, we can assume $N_{G_{n-3}}(w,C)=N_{G_{n-2}}(w,C)=U_1=\{u_1,u_3,\dots,u_{n-6}\}$, and $U_2=V(C)\setminus U_1=\{u_2,u_4,\dots,u_{n-5},u_{n-4}\}$. Now we prove that there exists a rainbow $k$-path joining $x$ and $y$ in $\mathbf{G}$ for every integer $k\in[4,n-1]$ by the following cases.

\vspace{6pt}
$\mathbf{Case~1.}$ $k=4$.
\vspace{6pt}

If there exists $u_i\in N_{G_{n-1}}(x)\cap U_1$, then $xu_iwy$ is the desired rainbow $4$-path with $xu_i\in E(G_{n-1})$, $u_iw\in E(G_{n-3})$ and $wy\in E(G_{n-2})$.

If $N_{G_{n-1}}(x)\cap U_1=\emptyset$, then $N_{G_{n-1}}(x)\subseteq U_2\cup\{y,w\}$ since $xz\notin E(G_{n-1})$. Moreover, since $d_{G_{n-1}}(x)\geq \frac{n+1}{2}$ and $|U_2\cup\{y,w\}|=\frac{n-3}{2}+2=\frac{n+1}{2}$, we have $N_{G_{n-1}}(x)=U_2\cup\{y,w\}$. Now we show that there exists a rainbow $4$-path by the following cases.

\begin{itemize}
\item There exists $u_i\in U_1$ such that $u_i\in N_{G_{n-3}}(y)$. Then $xu_{i+1}u_iy$ is the desired rainbow $4$-path with $xu_{i+1}\in E(G_{n-1})$, $u_{i+1}u_i\in E(G_i)$ and $u_iy\in E(G_{n-3})$.

\item $z\in N_{G_{n-3}}(y)$. Then $xwzy$ is the desired rainbow $4$-path with $xw\in E(G_{n-1})$, $wz\in E(G_{n-2})$ and $yz\in E(G_{n-3})$.

\item $N_{G_{n-3}}(y)\cap (U_1\cup \{z\})=\emptyset$. Then $N_{G_{n-3}}(y)\subseteq U_2\cup\{x,w\}$, and thus $N_{G_{n-3}}(y)=U_2\cup\{x,w\}$ since $d_{G_{n-3}}(y)\geq \frac{n+1}{2}$ and $|U_2\cup\{x,w\}|=\frac{n+1}{2}$. It follows that $xu_{n-4}u_{n-5}y$ is the desired rainbow $4$-path with $xu_{n-4}\in E(G_{n-1})$, $u_{n-5}u_{n-4}\in E(G_{n-5})$ and $u_{n-5}y\in E(G_{n-3})$.
\end{itemize}

$\mathbf{Case~2.}$ $k\in[5,n-1]$ and $\{u_{n-5},u_{n-4}\}\cap N_{G_{n-1}}(x)\neq\emptyset$.
\vspace{6pt}

Now we show $\mathbf{G}$ contains a rainbow $k$-path joining $x$ and $y$ if $u_{n-4}\in N_{G_{n-1}}(x)$, and the case for $u_{n-5}\in N_{G_{n-1}}(x)$ is similar and omitted.

Note that $u_1Cu_{n-5}$ is an alternating path of $U_1$ and $U_2$, i.e., every vertex on the path alternates between $U_1$ and $U_2$. Since $u_1\in U_1$ and $u_{n-5}\in U_2$, the two vertices at the same distance from $u_{n-4}$ on $u_{n-4}Cu_{n-5}$ and on $u_{n-4}C^-u_1$ belong to different sets $U_1$ and $U_2$. For example, the distance from $u_{n-4}$ to $u_{k-4}$ along $u_{n-4}Cu_{n-5}$ is $k-4$, and the distance to $u_{n-k}$ along $u_{n-4}C^-u_1$ is also $k-4$. Here either $u_{k-4}\in U_1$ and $u_{n-k}\in U_2$, or $u_{k-4}\in U_2$ and $u_{n-k}\in U_1$. Thus $\{u_{k-4},u_{n-k}\}\cap U_1\neq\emptyset$, it follows that $xu_{n-4}Cu_{k-4}wy$ is the desired rainbow $k$-path with $xu_{n-4}\in E(G_{n-1})$, $u_{k-4}w\in E(G_{n-3})$ and $wy\in E(G_{n-2})$ if $u_{k-4}\in U_1$, and $xu_{n-4}C^-u_{n-k}wy$ is the desired rainbow $k$-path with $xu_{n-4}\in E(G_{n-1})$, $u_{n-k}w\in E(G_{n-3})$ and $wy\in E(G_{n-2})$ if $u_{n-k}\in U_1$.

$\mathbf{Case~3.}$ $k\in[5,n-1]$ and $\{u_{n-4},u_{n-5}\}\cap N_{G_{n-1}}(x)=\emptyset$.
\vspace{6pt}

Now we show $\mathbf{G}$ contains a rainbow $k$-path joining $x$ and $y$.

Since $d_{G_{n-1}}(x)\geq \frac{n+1}{2}$ and $xz\notin E(G_{n-1})$, we obtain $d_{G_{n-1}}(x,C)\geq \frac{n-3}{2}$. Thus $N_{G_{n-1}}(x)\cap U_1\neq\emptyset$ and $N_{G_{n-1}}(x)\cap (U_2\setminus\{u_{n-4},u_{n-5}\})\neq\emptyset$ since $|U_1|=\frac{n-5}{2}$ and $|U_2\setminus\{u_{n-4},u_{n-5}\}|=\frac{n-7}{2}$.
Take $u_p \in N_{G_{n-1}}(x) \cap U_1$ and $u_q \in N_{G_{n-1}}(x) \cap (U_2 \setminus \{u_{n-4}, u_{n-5}\})$. Note that $|u_{\frac{n-5}{2}}Cu_{n-5}|=|u_{\frac{n-5}{2}}C^-u_{n-4}|=\frac{n-3}{2}$. Now we divide into four subcases based on the parity of $k-3$ and whether $u_{\frac{n-5}{2}}$ belongs to $U_1$ or $U_2$.

$\mathbf{Subcase~3.1.}$ $k-3$ is even and $u_{\frac{n-5}{2}}\in U_1$.

Clearly, $\frac{n-5}{2}$ is odd, then $k-3\in\left[2,\frac{n-3}{2}\right]\cup \left[\frac{n+1}{2},n-5\right]$ since $n$ is odd and $k-3$ is even.

\begin{itemize}
    \item $k-3\in\left[2,\frac{n-3}{2}\right]$. Then $\max\{|u_qCu_{n-5}|,|u_qC^-u_{n-4}|\}>\frac{n-3}{2}$ since $u_q\in U_2\setminus \{u_{n-4},$ $u_{n-5}\}$. Without loss of generality, we assume $|u_qCu_{n-5}|>\frac{n-3}{2}$. This means $u_qCu_{q+k-4}$ is an alternating path of $U_1$ and $U_2$. Since $|u_qCu_{q+k-4}|=k-3$ is even, $u_{q+k-4}\in U_1$, and thus $xu_qCu_{q+k-4}wy$ is the desired rainbow $k$-path with $xu_q\in E(G_{n-1})$, $u_{q+k-4}w\in E(G_{n-3})$ and $wy\in E(G_{n-2})$.

    \item $k-3\in\left[\frac{n+1}{2},n-5\right]$. Then $\min\{|u_pCu_{n-5}|,|u_pC^-u_{n-4}|\}\leq \frac{n-3}{2}$ since $u_p\in U_1$. Without loss of generality, we assume $|u_pCu_{n-5}|\leq \frac{n-3}{2}$. This means $\{u_{n-5},u_{n-4}\}\subseteq V(u_pCu_{p+k-4})$ and $p+k-4\neq n-4$. Moreover, since $|u_pCu_{p+k-4}|=k-3$ is even, we have $u_{p+k-4}\in U_1$ and thus $xu_pCu_{p+k-4}wy$ is the desired rainbow $k$-path.
\end{itemize}

$\mathbf{Subcase~3.2.}$ $k-3$ is even and $u_{\frac{n-5}{2}}\in U_2$.

Note that $\frac{n-5}{2}$ is even, then $k-3\in\left[2,\frac{n-5}{2}\right]\cup \left[\frac{n-1}{2},n-5\right]$.

\begin{itemize}
    \item $k-3\in\left[2,\frac{n-5}{2}\right]$. Then $\max\{|u_qCu_{n-5}|,|u_qC^-u_{n-4}|\}\geq\frac{n-3}{2}$ since $u_q\in U_2\setminus \{u_{n-4},$ $ u_{n-5}\}$. Without loss of generality, we assume $|u_qCu_{n-5}|\geq\frac{n-3}{2}$. Then $u_{q+k-4}\in V(u_qCu_{n-5})$ and $q+k-4\neq n-5$. Since $|u_qCu_{q+k-4}|=k-3$ is even, $u_{q+k-4}\in U_1$, and thus $xu_qCu_{q+k-4}wy$ is the desired rainbow $k$-path.

    \item $k-3\in\left[\frac{n-1}{2},n-5\right]$. Then $\min\{|u_pCu_{n-5}|,|u_pC^-u_{n-4}|\}< \frac{n-3}{2}$ since $u_p\in U_1$. Without loss of generality, we assume $|u_pCu_{n-5}|< \frac{n-3}{2}$. Then $\{u_{n-5},u_{n-4}\}\subseteq V(u_pCu_{p+k-4})$ and $p+k-4\neq n-4$. Since $|u_pCu_{p+k-4}|=k-3$ is even, we have $u_{p+k-4}\in U_1$ and thus $xu_pCu_{p+k-4}wy$ is the desired rainbow $k$-path.
\end{itemize}

$\mathbf{Subcase~3.3.}$ $k-3$ is odd and $u_{\frac{n-5}{2}}\in U_1$.

Clearly, $\frac{n-5}{2}$ is odd, then $k-3\in\left[3,\frac{n-5}{2}\right]\cup \left[\frac{n-1}{2},n-4\right]$.

\begin{itemize}
    \item $k-3\in\left[3,\frac{n-5}{2}\right]$. Then $\max\{|u_pCu_{n-5}|,|u_pC^-u_{n-4}|\}\geq\frac{n-3}{2}$. Without loss of generality, we assume $|u_pCu_{n-5}|\geq\frac{n-3}{2}$. Then $u_{p+k-4}\in U_1$, and thus $xu_pCu_{p+k-4}wy$ is the desired rainbow $k$-path.

    \item $k-3\in\left[\frac{n-1}{2},n-4\right]$. Then $\min\{|u_qCu_{n-5}|,|u_qC^-u_{n-4}|\}< \frac{n-3}{2}$. Without loss of generality, we assume $|u_qCu_{n-5}|< \frac{n-3}{2}$. Then $u_{q+k-4}\in U_1$ and thus $xu_qCu_{q+k-4}wy$ is the desired rainbow $k$-path.
\end{itemize}

$\mathbf{Subcase~3.4.}$ $k-3$ is odd and $u_{\frac{n-5}{2}}\in U_2$.

Note that $\frac{n-5}{2}$ is even, then $k-3\in\left[3,\frac{n-3}{2}\right]\cup \left[\frac{n+1}{2},n-4\right]$.

\begin{itemize}
    \item $k-3\in\left[3,\frac{n-3}{2}\right]$. Then $\max\{|u_pCu_{n-5}|,|u_pC^-u_{n-4}|\}>\frac{n-3}{2}$. Without loss of generality, we assume $|u_pCu_{n-5}|>\frac{n-3}{2}$.
        Then $u_{p+k-4}\in U_1$, and thus $xu_pCu_{p+k-4}wy$ is the desired rainbow $k$-path.

    \item $k-3\in\left[\frac{n+1}{2},n-4\right]$. Then $\min\{|u_qCu_{n-5}|,|u_qC^-u_{n-4}|\}\leq \frac{n-3}{2}$. Without loss of generality, we assume $|u_qCu_{n-5}|\leq \frac{n-3}{2}$. Then $u_{q+k-4}\in U_1$, and thus $xu_qCu_{q+k-4}wy$ is the desired rainbow $k$-path.
\end{itemize}

Combining the above arguments, $(\mathrm{R_1})$ holds, thus $\mathbf{G}$ is rainbow panconnected and Lemma \ref{lem3} holds.
\end{proof}

Based Lemmas \ref{lem2} and \ref{lem3}, we assume that $\mathbf{H_j}$ contains no rainbow Hamiltonian cycles or rainbow $(n-4)$-cycles for any $j\in[n-2]$ in the following. Thus $\mathbf{H}$ contains neither a rainbow $(n-3)$-cycle nor a rainbow $(n-4)$-cycle.

\begin{remark}\label{re1}
Recall that for each $j\in[n-2]$, $\delta(H_j)\geq\frac{n+1}{2}-3=\frac{n-5}{2}$ and
$\mathbf{H_j}=\mathbf{H}\setminus\{H_j\}$ is a collection of not necessarily distinct $(n-3)$-vertex graphs with the same vertex set $V\setminus\{x,y,z\}$.
By Corollary \ref{cor2-3}, one of the following three statements holds:
\begin{enumerate}
\item[\textup{(1)}] $\mathbf{H_j}$ has a rainbow Hamiltonian path;

\item[\textup{(2)}] $n-3$ is even and $\mathbf{H_j}$ consists of $n-3$ copies of $K_{\frac{n-3}{2}}\cup K_{\frac{n-3}{2}}$;

\item[\textup{(3)}] $n-3$ is even and there is a partition $(F,I)$ of $V\setminus\{x,y,z\}$ with $|F|=\frac{n-5}{2}$ and $|I|=\frac{n-1}{2}$. For each $i\in [n-2]\setminus\{j\}$, $H_i=H_i[F]\vee H_i[I]$, where $H_i[I]$ is an independent set and $H_i[F]$ is an arbitrary graph.
\end{enumerate}
\end{remark}

Next, Lemmas \ref{lem6}-\ref{lem8} study the above cases (1)-(3), respectively. To prove Lemma \ref{lem6}, we first present the following Lemma \ref{lem5}.

\begin{lemma}\label{lem5}
Let $P^*$ be a rainbow Hamiltonian path in $\mathbf{H}$ with endpoints $w_1$ and $w_2$, and $\phi\colon E(P^*)\to [n-2]$ be an injective and $\{f_1,f_2\}=[n-2]\setminus im(\phi)$, where $im(\phi)$ is the image of $\phi$. Then $n-5\leq d_{G_{f_1}}(w_1,P^*)+d_{G_{f_2}}(w_2,P^*)\leq n-4$. Furthermore, $\frac{n-5}{2}\leq d_{G_{f_1}}(w_1,P^*),d_{G_{f_2}}(w_2,P^*)\leq \frac{n-3}{2}$.
\end{lemma}

\begin{proof}
Without loss of generality, let $P^*=v_1v_2\cdots v_{n-3}$, where $v_1=w_1$ and $v_{n-3}=w_2$.
Then $v_1v_{n-4},v_1v_{n-3}\notin E(G_{f_1})$ and $v_1v_{n-3},v_2v_{n-3}\notin E(G_{f_2})$ since $\mathbf{H}$ contains neither a rainbow $(n-3)$-cycle nor a rainbow $(n-4)$-cycle.
Now we consider the following sets:
\[
I_{f_1}=\{i\in[1,n-6]:v_1v_{i+1}\in E(G_{f_1})\}, \quad I_{f_2}=\{i\in[3,n-4]:v_iv_{n-3}\in E(G_{f_2})\}.
\]

Clearly, $|I_{f_1}|=d_{G_{f_1}}(w_1,P^*)\geq d_{G_{f_1}}(v_1)-3\geq\frac{n-5}{2}$ and $|I_{f_2}|=d_{G_{f_2}}(w_2,P^*)\geq d_{G_{f_2}}(v_{n-3})-3\geq\frac{n-5}{2}$, which implies $|I_{f_1}|+|I_{f_2}|\geq n-5$.

If $I_{f_1}\cap I_{f_2}\neq\emptyset$, say, $i\in I_{f_1}\cap I_{f_2}$, then $v_1P^*v_iv_{n-3}P^*v_{i+1}v_1$ is a rainbow $(n-3)$-cycle in $\mathbf{H}$, a contradiction. Thus $I_{f_1}\cap I_{f_2}=\emptyset$, which implies $n-5\leq |I_{f_1}|+|I_{f_2}|=|I_{f_1}\cup I_{f_2}|\leq n-4$. Since $n$ is odd, we have $|I_{f_1}|,|I_{f_2}|\in\{\frac{n-5}{2},\frac{n-3}{2}\}$.
\end{proof}

For a set $A$, we denote $\max A = \max\{x: x\in A\}$ and $\min A = \min\{x: x\in A\}$.

\begin{lemma}\label{lem6}
If $\mathbf{H_j}$ has a rainbow Hamiltonian path for some $j\in[n-2]$, then $\mathbf{G}$ is rainbow panconnected.
\end{lemma}

\begin{proof}
Without loss of generality, let $j=n-2$ and $P=u_1u_2\cdots u_{n-3}$ be a rainbow Hamiltonian path in $\mathbf{H_{n-2}}$ with $u_iu_{i+1}\in E(G_i)$ for each $i\in [n-4]$. Note that $P$ has no edges from $G_{n-3}$ or $G_{n-2}$, then we define the following two sets:
\[
A_1=\{i\in[2,n-5]:u_1u_{i}\in E(G_{n-3})\}, \quad B_1=\{i\in[3,n-4]:u_iu_{n-3}\in E(G_{n-2})\}.
\]
Since $\mathbf{H}$ contains no rainbow $(n-4)$-cycles or $(n-3)$-cycles, we have $u_1u_{n-4}\notin E(G_{n-3})$, $u_1u_{n-3}\notin E(G_{n-3})\cup E(G_{n-2})$ and $u_2u_{n-3}\notin E(G_{n-2})$.
Thus $|A_1|=d_{G_{n-3}}(u_1,P)$ and $|B_1|=d_{G_{n-2}}(u_{n-3},P)$.

By Lemma \ref{lem5}, we have $n-5\leq |A_1|+|B_1|\leq n-4$ and $|A_1|,|B_1|\in\{\frac{n-5}{2},\frac{n-3}{2}\}$. Without loss of generality, we assume that $|A_1|=\frac{n-5}{2}$, which implies $x,y,z\in N_{G_{n-3}}(u_1)$ since $\delta(G_{n-3})\geq\frac{n+1}{2}$.

Now we compute $|B_1|$.
Let $s=\min A_1$ and $t=\max A_1$. Then $2\leq s\leq t\leq n-5$. We partition $A_1$ into $l(\geq1)$ maximal consecutive integer subsets $D_1,D_2,\dots,D_l$ such that $\max D_p < \min D_q$ for all $p<q$ (each $D_i$ may consist of a single element). Let $b_{i-1}=\min D_i$ and  $a_i=\max D_i$ for each $i \in [l]$. Then $2\leq s=b_0\leq a_1<b_1\leq a_2<b_2\leq\dots<b_{l-1}\leq a_l=t\leq n-5$, and $b_i\geq a_i+2$ for any $i\in[l-1]$. In particular, $\{u_s,u_{s+1},\dots,u_{t-1}\}=\emptyset$ if $s=t$ (i.e., $n=7$ and $|A_1|=1$), $\{u_{b_{i-1}},u_{b_{i-1}+1},\dots,u_{a_{i}-1}\}=\emptyset$ if $|D_i|=1$, and $\{u_{a_i},u_{a_i+1},\dots,u_{b_i-3}\}=\emptyset$ if $b_i=a_i+2$.

By the definition of $B_1$, we have
\begin{equation}\tag{$2$}\label{e2}
|B_1|=d_{G_{n-2}}(u_{n-3},u_1Pu_{s-1})+d_{G_{n-2}}(u_{n-3},u_sPu_{t-1})+d_{G_{n-2}}(u_{n-3},u_tPu_{n-4}) .
\end{equation}

\begin{claimnum}\label{c1}
$d_{G_{n-2}}(u_{n-3},u_1Pu_{s-1})\leq
\begin{cases}
0, & \text{if } 2\leq s\leq 5;\\
s-5, & \text{if } s\geq 6.
\end{cases}$
\end{claimnum}

\begin{proof}
Clearly, $u_1,u_2\notin N_{G_{n-2}}(u_{n-3})$. For $s\geq4$, we show $N_{G_{n-2}}(u_{n-3})\cap\{u_{s-2},u_{s-1}\}=\emptyset$.

If $u_{s-2}\in N_{G_{n-2}}(u_{n-3})$, then $u_1u_sPu_{n-3}u_{s-2}Pu_1$ is a rainbow $(n-4)$-cycle in $\mathbf{H}$ with $u_1u_s\in E(G_{n-3})$ and $u_{n-3}u_{s-2}\in E(G_{n-2})$, a contradiction.

If $u_{s-1}\in N_{G_{n-2}}(u_{n-3})$, then $u_1u_sPu_{n-3}u_{s-1}Pu_1$ is a rainbow $(n-3)$-cycle in $\mathbf{H}$ with $u_1u_s\in E(G_{n-3})$ and $u_{n-3}u_{s-1}\in E(G_{n-2})$, a contradiction.

Therefore, $d_{G_{n-2}}(u_{n-3},u_1Pu_{s-1})=0$ for $2\leq s\leq 5$ and  $d_{G_{n-2}}(u_{n-3},u_1Pu_{s-1})\leq s-5$ for $s\geq 6$.
\end{proof}

\begin{claimnum}\label{c2}
$d_{G_{n-2}}(u_{n-3},u_sPu_{t-1})\leq
\begin{cases}
0, & \text{if } l=1;\\
t-s-l-\frac{n-9}{2}, & \text{if } l\geq 2.
\end{cases}$
\end{claimnum}

\begin{proof}
If $l=1$, then $N_{G_{n-2}}(u_{n-3})\cap\{u_s,u_{s+1},\dots,u_{t-1}\}=\emptyset$. Otherwise, there exists some $i\in\{s,s+1,\dots,t-1\}$ such that $u_{n-3}u_{i}\in E(G_{n-2})$, then $u_1u_{i+1}Pu_{n-3}u_iPu_1$ is a rainbow $(n-3)$-cycle in $\mathbf{H}$ with $u_1u_{i+1}\in E(G_{n-3})$ and $u_{n-3}u_i\in E(G_{n-2})$, a contradiction. Thus $d_{G_{n-2}}(u_{n-3},u_sPu_{t-1})=0$.

If $l\geq2$, then

\begin{itemize}
\item $N_{G_{n-2}}(u_{n-3})\cap\{u_{b_{i-1}},u_{b_{i-1}+1},\dots,u_{a_{i}-1}\}=\emptyset$ for each $i\in[l]$. If not, there exists some $p\in\{b_{i-1},b_{i-1}+1,\dots,a_{i}-1\}$ such that $u_{n-3}u_p\in E(G_{n-2})$, thus $u_1u_{p+1}Pu_{n-3}u_pPu_1$ is a rainbow $(n-3)$-cycle in $\mathbf{H}$ with $u_1u_{p+1}\in E(G_{n-3})$ and $u_{n-3}u_p\in E(G_{n-2})$, a contradiction.
\item $N_{G_{n-2}}(u_{n-3})\cap\{u_{b_i-2},u_{b_i-1}\}=\emptyset$ for each $i\in[l-1]$. Otherwise, if $u_{n-3}u_{b_i-2}\in E(G_{n-2})$, then $u_1u_{b_i}Pu_{n-3}u_{b_i-2}Pu_1$ is a rainbow $(n-4)$-cycle in $\mathbf{H}$, a contradiction; if $u_{n-3}u_{b_i-1}\in E(G_{n-2})$, then $u_1u_{b_i}Pu_{n-3}u_{b_i-1}Pu_1$ is a rainbow $(n-3)$-cycle in $\mathbf{H}$, a contradiction.
\end{itemize}

Therefore, $N_{G_{n-2}}(u_{n-3})\cap\{u_s,u_{s+1},\dots,u_{t-1}\}\subseteq \bigcup\limits_{i=1}^{l-1}\{u_{a_i},u_{a_i+1},\dots,u_{b_i-3}\}$, and thus $d_{G_{n-2}}(u_{n-3},u_sPu_{t-1})\leq \sum\limits_{i=1}^{l-1}(b_i-a_i-2)$.

On the other hand, if $l\geq2$, by the definition of $A_1$ and the above notation, the number of vertices in $\{u_s,u_{s+1},\dots,u_{t}\}\setminus N_{G_{n-3}}(u_1)$
is
\begin{equation}\tag{$3$}\label{e3}
\sum\limits_{i=1}^{l-1}(b_i-a_i-1)=t-s+1-d_{G_{n-3}}(u_1,P)=t-s-\frac{n-7}{2}.
\end{equation}

By \eqref{e3}, we have
\begin{equation*}
d_{G_{n-2}}(u_{n-3},u_sPu_{t-1}) \leq \sum\limits_{i=1}^{l-1}(b_i-a_i-2)=t-s-l-\frac{n-9}{2}.
\end{equation*}

Therefore, Claim \ref{c2} holds.
\end{proof}

\begin{claimnum}\label{c3}
One of the following three cases holds: $(\mathrm{a})$ $s=3$, $l=1$; $(\mathrm{b})$ $s=2$, $l=2$; $(\mathrm{c})$ $s=2$, $l=1$.
\end{claimnum}

\begin{proof}
We first show $2\leq s\leq 5$. Suppose for contradiction that $s\geq6$.

If $l=1$, then $|A_1|=t-s+1=\frac{n-5}{2}$, which implies that $s-t=-\frac{n-7}{2}$. By Claims \ref{c1}-\ref{c2} and \eqref{e2}, we have
\begin{equation*}
|B_1|\leq (s-5)+(n-3-t)=\frac{n-9}{2},
\end{equation*}
a contradiction to $|B_1|\in\{\frac{n-5}{2},\frac{n-3}{2}\}$.

If $l\geq 2$, by Claims \ref{c1}-\ref{c2} and \eqref{e2}, we have
\begin{equation*}
|B_1|\leq (s-5)+(t-s-l-\frac{n-9}{2})+(n-3-t) \leq \frac{n-11}{2},
\end{equation*}
a contradiction to $|B_1|\in\{\frac{n-5}{2},\frac{n-3}{2}\}$.

Therefore, $2\leq s\leq 5$. By Claims \ref{c1}-\ref{c2} and \eqref{e2}, we obtain
\begin{equation}\tag{$4$}\label{eq:1}
\begin{aligned}
|B_1| &= d_{G_{n-2}}(u_{n-3},u_sPu_{t-1}) + d_{G_{n-2}}(u_{n-3},u_tPu_{n-4}) \\
&\leq
\begin{cases}
n-3-t=n-3-(\frac{n-7}{2}+s)=\frac{n+1}{2}-s, & \text{if } l=1; \\[4pt]
(t-s-l-\frac{n-9}{2})+(n-3-t)=\frac{n+3}{2}-s-l, & \text{if } l\geq 2.
\end{cases}
\end{aligned}
\end{equation}

Since $|B_1|\in\{\frac{n-5}{2},\frac{n-3}{2}\}$, by \eqref{eq:1}, there exist the following three cases: $(\mathrm{a})$ $s=3$, $l=1$; $(\mathrm{b})$ $s=2$, $l=2$; $(\mathrm{c})$ $s=2$, $l=1$.
\end{proof}

Combining the above arguments, to prove that $\mathbf{G}$ is rainbow panconnected, we only need to show $(\mathrm{R_1})$ holds for the following three cases: $(\mathrm{a})$ $s=3$, $l=1$; $(\mathrm{b})$ $s=2$, $l=2$; $(\mathrm{c})$ $s=2$, $l=1$. To this end, we show the following claim.

\begin{claimnum}\label{claim5}
For every integer $k\in [3,t+1]$, if $u_{k-1}\in N_{G_{n-3}}(u_1)$, then there exists a rainbow $k$-path joining $x$ and $y$ in $\mathbf{G}$.
\end{claimnum}

\begin{proof}
Choose an arbitrary integer $k\in [3,t+1]$ such that $u_{k-1}\in N_{G_{n-3}}(u_1)$, then $P'=u_{k-2}Pu_1u_{k-1}Pu_{n-3}$ is a rainbow Hamiltonian path in $\mathbf{H}$ with an injective $\varphi\colon E(P')\to [n-2]$ and $\{k-2,n-2\}=[n-2]\setminus im(\varphi)$. By Lemma \ref{lem5}, we obtain $d_{G_{k-2}}(u_{k-2},P')\in\{\frac{n-5}{2},\frac{n-3}{2}\}$, which implies $N_{G_{k-2}}(u_{k-2})\cap\{x,y\}\neq\emptyset$ since $\delta(G_{k-2})\geq\frac{n+1}{2}$. Recall that $\{x,y\}\subseteq N_{G_{n-3}}(u_1)$. Thus $xu_1Pu_{k-2}y$ is a rainbow $k$-path with $xu_1\in E(G_{n-3})$ and $u_{k-2}y\in E(G_{k-2})$ if $y\in N_{G_{k-2}}(u_{k-2})$; or $xu_{k-2}Pu_1y$ is a rainbow $k$-path with $xu_{k-2}\in E(G_{k-2})$ and $u_1y\in E(G_{n-3})$ if $x\in N_{G_{k-2}}(u_{k-2})$.
\end{proof}

Now we show that $(\mathrm{R_1})$ holds for all three cases $(\mathrm{a})$, $(\mathrm{b})$ and $(\mathrm{c})$.

$\mathbf{Case~1.}$ $s=3$, $l=1$.
\vspace{6pt}

Together with $|A_1|=\frac{n-5}{2}$ and the definition of $s,t,l$, we obtain $t=\frac{n-1}{2}$ and $N_{G_{n-3}}(u_1)=\{u_3,u_4,\dots,u_{\frac{n-1}{2}},x,y,z\}$.
Substituting the values of $s,t$ into \eqref{eq:1}, we have $|B_1|\leq \frac{n-5}{2}$, which together with $|B_1|\in\{\frac{n-5}{2}, \frac{n-3}{2}\}$ yields $|B_1| = \frac{n-5}{2}$. It follows that $N_{G_{n-2}}(u_{n-3})=\{u_{\frac{n-1}{2}},u_{\frac{n+1}{2}},\dots,u_{n-4},x,y,z\}$ since $\delta(G_{n-2})\geq\frac{n+1}{2}$.

Let $P''=u_1u_2\cdots u_{n-3}$ be a rainbow Hamiltonian path with $u_iu_{i+1}\in E(G_i)$ for each $i\in [n-5]$ and $u_{n-4}u_{n-3}\in E(G_{n-2})$. Note that $P''$ has no edges from $G_{n-4}$ or $G_{n-3}$. Applying the same discussion to $P''$ as to $P$, we have $N_{G_{n-4}}(u_{n-3})=N_{G_{n-2}}(u_{n-3})=\{u_{\frac{n-1}{2}},u_{\frac{n+1}{2}},\dots,u_{n-4},x,y,z\}$.

If $k=n-1$, then $xu_1Pu_{n-3}y$ is the desired rainbow $k$-path with $xu_1\in E(G_{n-3})$ and $u_{n-3}y\in E(G_{n-2})$.

If $k\in\left[4,\frac{n+1}{2}\right]$, then $u_{k-1}\in \{u_3,u_4,\dots,u_{\frac{n-1}{2}}\}\subseteq N_{G_{n-3}}(u_1)$. By Claim \ref{claim5}, there exists a rainbow $k$-path joining $x$ and $y$.

If $k\in\left[\frac{n+3}{2},n-2\right]$, then $u_{k-2}\in \{u_{\frac{n-1}{2}},u_{\frac{n+1}{2}},\dots,u_{n-4}\}\subseteq N_{G_{n-2}}(u_{n-3})$. We can claim that $N_{G_{n-3}}(u_2)\cap\{u_{\frac{n+1}{2}},u_{\frac{n+3}{2}},\dots,u_{n-3}\}=\emptyset$. Otherwise, if $u_2u_i\in E(G_{n-3})$ for some $i\in\left[\frac{n+1}{2},n-3\right]$, then $u_2u_iPu_{n-3}u_{i-1}Pu_2$ is a rainbow $(n-4)$-cycle in $\mathbf{H}$ with $u_2u_i\in E(G_{n-3})$ and $u_{n-3}u_{i-1}\in E(G_{n-2})$, a contradiction. Therefore, we have
\[
d_{G_{n-3}}(u_2,P) \leq n-3-|\{u_{\frac{n+1}{2}},u_{\frac{n+3}{2}},\dots,u_{n-3}\}|-1= \frac{n-3}{2},
\]
which implies $N_{G_{n-3}}(u_2)\cap\{x,y\}\neq\emptyset$ since $\delta(G_{n-3})\geq \frac{n+1}{2}$. It follows that $xu_2Pu_{k-2}u_{n-3}y$ is the desired rainbow $k$-path with $xu_2\in E(G_{n-3})$, $u_{k-2}u_{n-3}\in E(G_{n-2})$ and $u_{n-3}y\in E(G_{n-4})$ if $x\in N_{G_{n-3}}(u_2)$; or $xu_{n-3}u_{k-2}Pu_2y$ is the desired rainbow $k$-path with $xu_{n-3}\in E(G_{n-4})$, $u_{n-3}u_{k-2}\in E(G_{n-2})$ and $u_2y\in E(G_{n-3})$ if $y\in N_{G_{n-3}}(u_2)$.

Combining the above arguments, there exists a rainbow $k$-path joining $x$ and $y$ in $\mathbf{G}$ for every integer $k\in[4,n-1]$, i.e., $(\mathrm{R_1})$ holds.

\vspace{6pt}
$\mathbf{Case~2.}$ $s=2$, $l=2$.
\vspace{6pt}

In this case, $D_1=[2,a_1]$ and $D_2=[b_1,t]$, which implies $N_{G_{n-3}}(u_1)=\{u_2,\dots,u_{a_1}\}\cup\{u_{b_1},\dots,u_{t}\}\cup\{x,y,z\}$.
Note that $|\{a_1,\dots,b_1-3\}\cup \{t,\dots,n-4\}|=\frac{n-5}{2}$ since $|A_1|=a_1-b_1+t=\frac{n-5}{2}$,
$B_1\subseteq \{a_1,\dots,b_1-3\}\cup \{t,\dots,n-4\}$ and $|B_1|\in \{\frac{n-5}{2},\frac{n-3}{2}\}$, we have $B_1= \{a_1,\dots,b_1-3\}\cup \{t,\dots,n-4\}$ and thus $N_{G_{n-2}}(u_{n-3})=\{u_{a_1},\dots,u_{b_1-3}\}\cup\{u_t,\dots,u_{n-4}\}\cup\{x,y,z\}$.

Similar to the discussion of $P''$ in Case 1, we can obtain $N_{G_{n-4}}(u_{n-3})=N_{G_{n-2}}(u_{n-3})$.

If $k\in[3,a_1+1]\cup[b_1+1,t+1]$, then $u_{k-1}\in N_{G_{n-3}}(u_1)$. By Claim \ref{claim5}, there exists a rainbow $k$-path joining $x$ and $y$.

If $k\in[t+3,n-1]$, then $u_{k-3}\in N_{G_{n-2}}(u_{n-3})$, and thus $xu_1Pu_{k-3}u_{n-3}y$ is the desired rainbow $k$-path with $xu_1\in E(G_{n-3})$, $u_{k-3}u_{n-3}\in E(G_{n-2})$ and $u_{n-3}y\in E(G_{n-4})$.

If $k\in[a_1+3,b_1]$, then $u_{k-3}\in N_{G_{n-2}}(u_{n-3})$, and thus $xu_1Pu_{k-3}u_{n-3}y$ is the desired rainbow $k$-path with $xu_1\in E(G_{n-3})$, $u_{k-3}u_{n-3}\in E(G_{n-2})$ and $u_{n-3}y\in E(G_{n-4})$.

If $k=a_1+2$ with $a_1<b_1-2$ or $k=t+2$, then $u_{k-2}\in N_{G_{n-2}}(u_{n-3})$.
We claim $N_{G_{n-3}}(u_2)\cap(\{u_{a_1+1},\dots,u_{b_1-2}\}\cup\{u_{t+1},\dots,u_{n-3}\})=\emptyset$. Otherwise, there exists some $i\in[a_1+1,b_1-2]\cup[t+1,n-3]$ such that $u_2u_i\in E(G_{n-3})$. Note that $u_{i-1}u_{n-3}\in E(G_{n-2})$, thus $u_2u_iPu_{n-3}u_{i-1}Pu_2$ is a rainbow $(n-4)$-cycle in $\mathbf{H}$, a contradiction. In addition, $|\{u_{a_1+1},\dots,u_{b_1-2}\}\cup\{u_{t+1},\dots,u_{n-3}\}|=|\{u_{a_1},\dots,u_{b_1-3}\}\cup\{u_t,\dots,u_{n-4}\}|=d_{G_{n-2}}(u_{n-3},P)$.
Thus
\[
d_{G_{n-3}}(u_2,P) \leq n-3-d_{G_{n-2}}(u_{n-3},P)-1= \frac{n-3}{2}.
\]
Since $\delta(G_{n-3})\geq\frac{n+1}{2}$, $N_{G_{n-3}}(u_2)\cap\{x,y\}\neq\emptyset$. It follows that $xu_2Pu_{k-2}u_{n-3}y$ is the desired rainbow $k$-path with $xu_2\in E(G_{n-3})$, $u_{k-2}u_{n-3}\in E(G_{n-2})$ and $u_{n-3}y\in E(G_{n-4})$ if $x\in N_{G_{n-3}}(u_2)$; or $xu_{n-3}u_{k-2}Pu_2y$ is the desired rainbow $k$-path with $xu_{n-3}\in E(G_{n-4})$, $u_{k-2}u_{n-3}\in E(G_{n-2})$ and $u_2y\in E(G_{n-3})$ if $y\in N_{G_{n-3}}(u_2)$.

If $k=a_1+2$ with $a_1=b_1-2$, then $N_{G_{n-2}}(u_{n-3},P)=\{u_t,\dots,u_{n-4}\}$. Since $|B_1|=d_{G_{n-2}}(u_{n-3},P)=\frac{n-5}{2}$, we have $t=\frac{n-1}{2}$. It follows from $4\leq k=b_1\leq t=\frac{n-1}{2}$ that $n-k-1\in\left[\frac{n-1}{2},n-5\right]$, which implies $u_{n-k-1}\in N_{G_{n-2}}(u_{n-3})$. Note that $P'''=u_1Pu_{n-k-1}u_{n-3}Pu_{n-k}$ is a rainbow Hamiltonian path in $\mathbf{H}$ with an injective $\psi\colon E(P''')\to [n-2]$ and $\{n-3,n-k-1\}=[n-2]\setminus im(\psi)$. Combining this with Lemma \ref{lem5}, $d_{G_{n-k-1}}(u_{n-k},P''')\in\{\frac{n-5}{2},\frac{n-3}{2}\}$, and thus $N_{G_{n-k-1}}(u_{n-k})\cap\{x,y\}\neq\emptyset$ since $\delta(G_{n-k-1})\geq\frac{n+1}{2}$. It follows that $xu_{n-k}Pu_{n-3}y$ is the desired rainbow $k$-path with $xu_{n-k}\in E(G_{n-k-1})$ and $u_{n-3}y\in E(G_{n-4})$ if $x\in N_{G_{n-k-1}}(u_{n-k})$; or $xu_{n-3}Pu_{n-k}y$ is the desired rainbow $k$-path with $xu_{n-3}\in E(G_{n-4})$ and $u_{n-k}y\in E(G_{n-k-1})$ if $y\in N_{G_{n-k-1}}(u_{n-k})$.

Combining the above arguments, $(\mathrm{R_1})$ holds.

\vspace{6pt}
$\mathbf{Case~3.}$ $s=2$, $l=1$.
\vspace{6pt}

It immediately follows that $N_{G_{n-3}}(u_1)=\{u_2,u_3,\dots,u_{\frac{n-3}{2}},x,y,z\}$ and $N_{G_{n-2}}(u_{n-3},P)\subseteq\{u_{\frac{n-3}{2}},u_{\frac{n-1}{2}},\dots,u_{n-4}\}$. Since $|B_1|\in\{\frac{n-5}{2},\frac{n-3}{2}\}$, we distinguish the following two subcases in terms of the value of $|B_1|$.

$\mathbf{Subcase~3.1.}$ $|B_1|=\frac{n-3}{2}$.

In this subcase, $N_{G_{n-2}}(u_{n-3},P)=\{u_{\frac{n-3}{2}},u_{\frac{n-1}{2}},\dots,u_{n-4}\}$. Now we show $n\geq 9$. Otherwise, we have $n=7$ and $u_{\frac{n-3}{2}}=u_2$, then $u_{n-3}u_2Pu_{n-3}$ is a rainbow $n-4$ cycle in $\mathbf{H}$ with $u_{n-3}u_2\in E(G_{n-2})$, a contradiction.

Note that $P=u_1u_2\cdots u_{n-3}$ is also a rainbow Hamiltonian path with $u_1u_2\in E(G_{n-3})$ and $u_iu_{i+1}\in E(G_i)$ for each $i\in [2,n-4]$. We can obtain $N_{G_1}(u_1,P)\subseteq\{u_2,u_3,\dots,u_{\frac{n-3}{2}}\}$. If not, there exists some $p\in[\frac{n-1}{2},n-3]$ such that $u_1u_p\in E(G_1)$, then $u_1u_pPu_{n-3}u_{p-1}Pu_1$ is a rainbow $(n-3)$-cycle in $\mathbf{H}$ with $u_1u_p\in E(G_1)$, $u_{n-3}u_{p-1}\in E(G_{n-2})$, $u_2u_1\in E(G_{n-3})$ and $u_iu_{i+1}\in E(G_i)$ for each $i\in[2,n-4]\setminus\{p-1\}$, a contradiction. Since $\delta(G_1)\geq\frac{n+1}{2}$, we have $N_{G_1}(u_1)=\{u_2,u_3,\dots,u_{\frac{n-3}{2}},x,y,z\}$.

Similar to the discussion of $P''$ in Case 1, we can obtain $N_{G_{n-4}}(u_{n-3},P)\subseteq N_{G_{n-2}}(u_{n-3})$, and $d_{G_{n-4}}(u_{n-3},P)\in\{\frac{n-5}{2},\frac{n-3}{2}\}$, which implies $N_{G_{n-4}}(u_{n-3})\cap\{x,y\}\neq\emptyset$. Without loss of generality, we assume $y\in N_{G_{n-4}}(u_{n-3})$.

If $k\in\left[3,\frac{n-1}{2}\right]$, then $u_{k-1}\in N_{G_{n-3}}(u_1)$. By Claim \ref{claim5}, there exists a rainbow $k$-path joining $x$ and $y$.

If $k\in\left[\frac{n+1}{2},n-2\right]$, then $u_{k-2}\in N_{G_{n-2}}(u_{n-3})$. Thus $xu_1u_3Pu_{k-2}u_{n-3}y$ is the desired rainbow $k$-path with $xu_1\in E(G_{n-3})$, $u_1u_3\in E(G_1)$, $u_{k-2}u_{n-3}\in E(G_{n-2})$ and $u_{n-3}y\in E(G_{n-4})$.

If $k=n-1$, then $xu_1Pu_{n-3}y$ is the desired rainbow $k$-path with $xu_1\in E(G_{n-3})$ and $u_{n-3}y\in E(G_{n-2})$.

Therefore, $(\mathrm{R_1})$ holds.

$\mathbf{Subcase~3.2.}$ $|B_1|=\frac{n-5}{2}$.

It is clear that there exists a unique $q\in\left[\frac{n-3}{2},n-4\right]$ such that $u_qu_{n-3}\notin E(G_{n-2})$.

If $q=n-4$, then this case coincides with Case 1 by setting $w_i=u_{n-2-i}$ for each $i\in[n-3]$ and replacing $u_i$ of Case 1 with $w_i$.
If $q\in\left[\frac{n-1}{2},n-5\right]$, then this case coincides with the subcase $b_i=a_i+2$ in Case 2.
By the same arguments as in Cases 1 and 2, we obtain a rainbow $k$-path joining $x$ and $y$ in $\mathbf{G}$, so we omit the detailed proof. It suffices to consider $q=\frac{n-3}{2}$, and we have $N_{G_{n-2}}(u_{n-3})=\{u_{\frac{n-1}{2}},\dots,u_{n-4},x,y,z\}$.

Since $u_{n-4}u_{n-3}\in E(G_{n-2})$, we can take $P''=u_1u_2\cdots u_{n-3}$ be a rainbow Hamiltonian path with $u_iu_{i+1}\in E(G_i)$ for each $i\in [n-5]$ and $u_{n-4}u_{n-3}\in E(G_{n-2})$ similar to the proof of Case 1. By Lemma \ref{lem5}, $d_{G_{n-4}}(u_{n-3},P)\in\{\frac{n-5}{2},\frac{n-3}{2}\}$.
If $d_{G_{n-4}}(u_{n-3},P)=\frac{n-3}{2}$, a similar analysis to that in Subcase 3.1 implies the existence of a rainbow $k$-path joining $x$ and $y$.
Now we assume that $d_{G_{n-4}}(u_{n-3},P)=\frac{n-5}{2}$. Then $x,y\in N_{G_{n-4}}(u_{n-3})$.

If $k\in\left[3,\frac{n-1}{2}\right]$, then $u_{k-1}\in N_{G_{n-3}}(u_1)$, and thus there exists a rainbow $k$-path joining $x$ and $y$ by Claim \ref{claim5}.

If $k\in\left[\frac{n+5}{2},n-1\right]$, then $u_{k-3}\in N_{G_{n-2}}(u_{n-3})$, and thus $xu_1Pu_{k-3}u_{n-3}y$ is the desired rainbow $k$-path with $xu_1\in E(G_{n-3})$, $u_{k-3}u_{n-3}\in E(G_{n-2})$ and $u_{n-3}y\in E(G_{n-4})$.

Now we consider $k\in\{\frac{n+1}{2},\frac{n+3}{2}\}$. Since $P=u_1u_2\cdots u_{n-3}$ is also a rainbow Hamiltonian path with $u_1u_2\in E(G_{n-3})$ and $u_iu_{i+1}\in E(G_i)$ for each $i\in [2,n-4]$, by Lemma \ref{lem5}, we have $d_{G_1}(u_1,P)\in\{\frac{n-5}{2},\frac{n-3}{2}\}$, and thus $N_{G_1}(u_1)\cap\{x,y\}\neq\emptyset$. We assume $xu_1\in E(G_1)$. The case $yu_1\in E(G_1)$ is similar, so we omit the details.

If $k=\frac{n+1}{2}$, then $xu_1u_{\frac{n-3}{2}}Pu_{n-6}u_{n-3}y$ is the desired rainbow $k$-path with $xu_1\in E(G_1)$, $u_1u_{\frac{n-3}{2}}\in E(G_{n-3})$, $u_{n-6}u_{n-3}\in E(G_{n-2})$ and $u_{n-3}y\in E(G_{n-4})$ for $n>9$; and $xu_1zu_6y$ is the desired rainbow $k$-path with $xu_1\in E(G_1)$, $u_1z\in E(G_6)$, $zu_6\in E(G_7)$ and $u_6y\in E(G_5)$ for $n=9$. For $n=7$, we have $N_{G_5}(u_2)\cap\{x,y\}\neq\emptyset$ since $u_2u_4\notin E(G_5)$ and $\delta(G_5)\geq 4$, it follows that $xu_2u_1y$ is the desired rainbow $k$-path with $xu_2\in E(G_5)$, $u_2u_1\in E(G_1)$ and $u_1y\in E(G_4)$ if $x\in N_{G_5}(u_2)$; or $xu_1u_2y$ is the desired rainbow $k$-path with $xu_1\in E(G_4)$, $u_1u_2\in E(G_1)$ and $u_2y\in E(G_5)$ if $y\in N_{G_5}(u_2)$.

If $k=\frac{n+3}{2}$, then $xu_1u_{\frac{n-3}{2}}Pu_{n-5}u_{n-3}y$ is the desired rainbow $k$-path with $xu_1\in E(G_1)$, $u_1u_{\frac{n-3}{2}}\in E(G_{n-3})$, $u_{n-5}u_{n-3}\in E(G_{n-2})$ and $u_{n-3}y\in E(G_{n-4})$ for $n>7$; and $xu_1zu_4y$ is the desired rainbow $k$-path with $xu_1\in E(G_1)$, $u_1z\in E(G_4)$, $zu_4\in E(G_5)$ and $u_4y\in E(G_3)$ for $n=7$.

Combining the above arguments, $(\mathrm{R_1})$ holds.

Therefore, $\mathbf{G}$ is rainbow panconnected, and the proof is complete.
\end{proof}

\begin{lemma}\label{lem7}
If $\mathbf{H_j}$ consists of $n-3$ copies of $K_{\frac{n-3}{2}}\cup K_{\frac{n-3}{2}}$ for some $j\in[n-2]$, then $\mathbf{G}$ is rainbow panconnected.
\end{lemma}

\begin{proof}
Since $H_i=G_i-\{x,y,z\}$ for each $i\in[n-2]$ and $\mathbf{H_j}=\mathbf{H}\setminus\{H_j\}$ for $j\in[n-2]$, we can assume that $V\setminus\{x,y,z\}=U_1\cup U_2$ such that $G_{i}[U_1]=K_{\frac{n-3}{2}}$ and $G_{i}[U_2]=K_{\frac{n-3}{2}}$ for each $i\in[n-2]\setminus\{j\}$, where $U_1=\{u_1,u_2,\dots,u_{\frac{n-3}{2}}\}$ and $U_2=\{u_{\frac{n-1}{2}},u_{\frac{n+1}{2}},\dots,u_{n-3}\}$.
Since $\delta(G_i)\geq \frac{n+1}{2}$ for each $i\in[n-2]\setminus\{j\}$, we have $\{x,y,z\}\subset N_{G_i}(u)$ for any $u\in U_1\cup U_2$.
Since $xz\notin E(G_{n-1})$ and $\delta(G_{n-1})\geq \frac{n+1}{2}$, we can obtain $E_{G_{n-1}}[\{x\},U_i]\neq\emptyset$ for each $i\in\{1,2\}$.

If $k-2\in\left[2,\frac{n-3}{2}\right]$, then $P'=xu_1u_2\cdots u_{k-2}y$ is the desired rainbow $k$-path with an injection $\phi\colon E(P')\to [n-2]\setminus\{j\}$.
Next, we consider $k-2\in\left[\frac{n-1}{2},n-3\right]$.

If $E_{G_j}[U_1,U_2]\neq\emptyset$, without loss of generality, we assume $u_{\frac{n-3}{2}}u_{\frac{n-1}{2}}\in E(G_j)$. Then $xu_1u_2\dots u_{\frac{n-3}{2}}u_{\frac{n-1}{2}}\cdots u_{k-2}y$ is the desired rainbow $k$-path with $u_{\frac{n-3}{2}}u_{\frac{n-1}{2}}\in E(G_j)$.

If $E_{G_j}[U_1,U_2]=\emptyset$, then $G_{j}[U_1]=K_{\frac{n-3}{2}}$, $G_{j}[U_2]=K_{\frac{n-3}{2}}$ and $zu\in E(G_j)$ for any $u\in U_1\cup U_2$ since $\delta(G_j)\geq \frac{n+1}{2}$. If $k-2=\frac{n-1}{2}$, then $xu_1u_2\cdots u_{\frac{n-5}{2}}zu_{\frac{n-1}{2}}y$ is the desired rainbow $k$-path with $u_{\frac{n-5}{2}}z\in E(G_j)$. If $k-2\in \left[\frac{n+1}{2},n-3\right]$, then $xu_1u_2\cdots u_{\frac{n-3}{2}}zu_{\frac{n-1}{2}}\cdots u_{k-3}y$ is the desired rainbow $k$-path with $u_{\frac{n-3}{2}}z\in E(G_j)$.

Therefore, $(\mathrm{R_1})$ holds, and thus $\mathbf{G}$ is rainbow panconnected.
\end{proof}

\begin{lemma}\label{lem8}
Suppose that there is a partition $(F,I)$ of $V\setminus\{x,y,z\}$ with $|F|=\frac{n-5}{2}$ and $|I|=\frac{n-1}{2}$. For each $i\in[n-2]$, $H_i=H_i[F]\vee H_i[I]$, where $H_i[I]$ is an independent set and $H_i[F]$ is an arbitrary graph. Then $\mathbf{G}$ is rainbow panconnected or $\mathbf{G}=\mathbf{F_{n-1}}$.
\end{lemma}

\begin{proof}
Let $I=\{w_1,w_2,\dots,w_{\frac{n-1}{2}}\}$ and $F=\{v_1,v_2,\dots,v_{\frac{n-5}{2}}\}$.
For each $i\in[n-2]$, since $\delta(G_i)\geq \frac{n+1}{2}$ and $H_i[I]$ is an independent set, we have $G_i[I,F\cup\{x,y,z\}]=K_{\frac{n-1}{2},\frac{n+1}{2}}$. We distinguish the following two cases according to whether $E(G_{n-1}[I])=\emptyset$.

\vspace{6pt}
$\mathbf{Case~1.}$ $E(G_{n-1}[I])\neq\emptyset$.
\vspace{6pt}

Without loss of generality, we assume $w_{\frac{n-3}{2}}w_{\frac{n-1}{2}}\in E(G_{n-1})$. Let
\[
\overline{P}=
\begin{cases}
xw_1v_1\cdots w_{\frac{k-3}{2}}v_{\frac{k-3}{2}}w_{\frac{k-1}{2}}y, & \text{if } k \text{ is odd};\\[4pt]
xw_1v_1\cdots w_{\frac{k-4}{2}}v_{\frac{k-4}{2}} w_{\frac{n-3}{2}}w_{\frac{n-1}{2}}y, & \text{if } k(\geq6) \text{ is even};\\[4pt]
xw_{\frac{n-3}{2}}w_{\frac{n-1}{2}}y, & \text{if } k=4.
\end{cases}
\]
It is easy to verify $\overline{P}$ is the desired rainbow $k$-path for $k\in[4,n-1]$.
Thus $(\mathrm{R_1})$ holds.

\vspace{6pt}
$\mathbf{Case~2.}$ $E(G_{n-1}[I])=\emptyset$.
\vspace{6pt}

It immediately follows that $G_i=G_i[I]\vee G_i[F\cup\{x,y,z\}]$ for any $i\in[n-1]$ and $\delta(G_i[F\cup\{x,y,z\}])\geq 1$ since $|I|=\frac{n-1}{2}$ and $\delta(G_i)\geq\frac{n+1}{2}$ for each $i\in [n-1]$.

$\mathbf{Subcase~2.1.}$ There exists some $i\in[n-1]$ such that $E_{G_i}[F\cup\{z\},\{x,y\}]\neq\emptyset$.

Without loss of generality, we take $xv_1\in E(G_i)$. Let
\[
\overline{P}=
\begin{cases}
xw_1v_1\cdots w_{\frac{k-3}{2}}v_{\frac{k-3}{2}}w_{\frac{k-1}{2}}y, & \text{if } k \text{ is odd};\\[4pt]
xv_1w_1\cdots v_{\frac{k-2}{2}}w_{\frac{k-2}{2}}y \text{ with } xv_1\in E(G_i), & \text{if } k(\in[4,n-3]) \text{ is even};\\[4pt]
xv_1w_1\cdots v_{\frac{n-5}{2}}w_{\frac{n-5}{2}}zw_{\frac{n-3}{2}}y \text{ with } xv_1\in E(G_i), & \text{if } k=n-1.
\end{cases}
\]
One can readily check $\overline{P}$ is the desired rainbow $k$-path for $k\in[4,n-1]$.
Thus $(\mathrm{R_1})$ holds.

$\mathbf{Subcase~2.2.}$ $E_{G_i}[F\cup\{z\},\{x,y\}]=\emptyset$ for any $i\in[n-1]$.

Since $\delta(G_i[F\cup\{x,y,z\}])\geq 1$, we obtain that $xy$ is a component consisting of a single edge in $G_i[F\cup\{x,y,z\}]$. It follows that $\mathbf{G}=\mathbf{F_{n-1}}$ by Definition \ref{def1-2}, i.e., $(\mathrm{R_2})$ holds.

Therefore, $(\mathrm{R_1})$ or $(\mathrm{R_2})$ holds. The result thus follows.
\end{proof}

From Remark \ref{re1}, if there exists some $j\in[n-2]$ such that $\mathbf{H_j}=\mathbf{H}\setminus\{H_j\}$ satisfies (1) or (2) of Remark \ref{re1}, then $\mathbf{G}$ is rainbow panconnected by Lemma \ref{lem6} and Lemma \ref{lem7}, respectively. Otherwise, (3) of Remark \ref{re1} holds for every $j\in[n-2]$, i.e., for each $i\in[n-2]$, $H_i=H_i[F]\vee H_i[I]$, where $H_i[I]$ is an independent set and $H_i[F]$ is an arbitrary graph. In this case, either $\mathbf{G}$ is rainbow panconnected or $\mathbf{G}=\mathbf{F_{n-1}}$ by Lemma \ref{lem8}.

Combining with the above arguments, the proof of Theorem \ref{thm1-5} is completed.
{\hfill $\square$ \par}

\section*{Funding}
This work is supported by the National Natural Science Foundation of China (Grant Nos. 12371347, 12271337).

\section*{Declarations}
\noindent\textbf{Conflict of interest} The authors declare that they have no known competing financial interests or personal relationships that could have appeared to influence the work reported in this paper.\\
\textbf{Data availability} No data was used for the research described in the article.

\end{spacing}
\end{document}